\documentclass[12pt]{article}
\usepackage{amssymb}
\usepackage{array}
\usepackage{latexsym}
\usepackage{enumerate}
\usepackage{amsmath}
\usepackage{amsfonts}
\usepackage{amsthm}
\usepackage{graphicx}
\usepackage{psfrag}
\usepackage{comment}
\usepackage[english]{babel}
\usepackage[T1]{fontenc}
\usepackage{geometry}
\usepackage{hyperref}
\ifx\smallsetminus\undefined
\def\smallsetminus{\setminus}
\fi
\title{Orthogonal arrays from Hermitian varieties}
\author{A. Aguglia ${}^*$ \and L. Giuzzi
  \thanks{Research supported by  the Italian
    Ministry MIUR, Strutture geometriche, combinatoria e loro
    applicazioni.}}

\date{}

\theoremstyle{plain}
\newtheorem{prop}{Proposition}[section]
\newtheorem{theorem}[prop]{Theorem}

\newtheorem{lemma}[prop]{Lemma}
\theoremstyle{definition}
\newtheorem{remark}[prop]{Remark}

\def\cR{{\mathcal R}}
\def\cC{\mathcal C}
\newcommand{\cD}{\mathcal D}
\def\cB{\mathcal B}
\def\cW{\mathcal W}
\def\cP{\mathcal P}

\def\cS{\mathcal S}
\def\cA{\mathcal A}

\def\cH{\mathcal H}
\def\cQ{\mathcal Q}

\newcommand{\GF}{\mathrm{GF}}
\newcommand{\tr}{\mathrm{tr}}

\begin{document}

\maketitle
\begin{abstract}
A simple orthogonal array  $OA(q^{2n-1},q^{2n-2}, q,2)$  is
constructed by using  the action of a large subgroup of
$PGL(n+1,q^2)$ on a set of non--degenerate Hermitian varieties in
$PG(n,q^2)$.
\end{abstract}

\section{Introduction}

Let $\cQ=\{0,1,\ldots,q-1\}$ be a set of $q$ symbols and consider a
$(k\times N)$--matrix $A$ with entries in $\cQ$. The matrix $A$ is
an \emph{orthogonal array} with $q$ levels and strength $t$, in
short an $OA(N,k,q,t)$, if any $(t\times N)$--subarray of $A$
contains each $t\times 1$--column  with entries in $\cQ$, exactly
$\mu=N/q^t$ times. The number $\mu$ is called the \emph{index} of
the array $A$. An orthogonal array is \emph{simple} when it does not
contain any repeated column.

Orthogonal arrays were first considered in the early Forties, see
Rao \cite{Rao1, Rao2}, and have been intensively studied ever since,
see \cite{Sloane}. They have  been widely used in statistic, computer
science and cryptography.

There are also remarkable links between these arrays and affine
designs, see \cite{Sh,To}. In particular, an $OA(q\mu_1,k,q,1)$
exists if and only if there is a resolvable $1-(q\mu_1,\mu_1,k)$
design. Similarly, the existence of an $OA(q^2\mu_2,k,q,2)$, is
equivalent to that of  an affine $1-(q^2\mu_2,q\mu_2,k)$  design,
see \cite{Sh}




%
%
A general procedure for constructing an orthogonal array depends on
homogeneous forms $f_1,\ldots,f_k$,  defined over a subset
$\cW\subseteq\GF(q)^{n+1}$. The array
\[
A(f_1,\ldots,f_k;\cW)= \left\{\begin{pmatrix}
    f_1(x)\\
    f_2(x)\\
    \vdots\\
    f_k(x)
\end{pmatrix} : x \in \cW \right\}, \]
with  an arbitrary order of columns, provides an orthogonal array if
the size  of the intersection $V(f_i)\cap V(f_j)\cap\cW$  for
distinct varieties $V(f_i)$ and $V(f_j)$, is independent of the
choice of $i$, $j$. Here $V(f)$ denotes the algebraic variety
associated to $f$. This procedure was applied for linear functions
by Bose \cite{Bo}, and for quadratic functions by Fuji-Hara and
Miyamoto \cite{FuMi1,FuMi2}.

In this paper, we  construct a simple orthogonal array
$\cA_0=OA(q^{2n-1},q^{2n-2}, q,2)$ by using  the above procedure for
 Hermitian forms. To do this we look into the action of a large
subgroup of $PGL(n+1,q^2)$ on a set of non--degenerate Hermitian
varieties in $PG(n,q^2)$. The resulting orthogonal array $\cA_0$ is
closely related to an affine
$2-(q^{(2n-1)},q^{2(n-1)},q^{(2n-3)}+\ldots+q+1)$ design $\cS$, that
for $q\geq 2$, provides a non--classical model of the
$(2n-1)$--dimensional affine space $AG(2n-1,q)$. Precisely, the
points of $\cS$ are labelled by the columns of $\cA_0$, some
parallel classes of $\cS$ correspond to  the rows of $\cA_0$ and
each of the $q$ parallel blocks associated to a  given row of
$\cA_0$ is labelled by one of the $q$ different symbols in that row.

\section{Preliminary results on Hermitian varieties}
\label{cones}

Let $\Sigma=PG(n,q^2)$ be the desarguesian projective space of
dimension $n$ over $\GF(q^2)$ and denote by
$X=(x_1,x_2,\ldots,x_{n+1})$ homogeneous coordinates for its points.
The hyperplane $\Sigma_{\infty}: X_{n+1}=0$ will be taken as the
hyperplane at infinity.

We use $\sigma$ to write the involutory automorphism of $\GF(q^2)$
which leaves all the elements of the subfield $\GF(q)$ invariant. A
Hermitian variety $\cH(n,q^2)$ is the set of all points $X$ of
$\Sigma$  which are self conjugate under a Hermitian polarity $h$.
If $H$ is the Hermitian $(n+1)\times(n+1)$--matrix associated to
$h$, then the Hermitian variety  $\cH(n,q^2)$ has equation
\[X H(X^{\sigma})^T=0.\]
When $A$ is non--singular, the corresponding Hermitian variety is
\emph{non--degenerate}, whereas if $A$ has rank $n$, the related
variety is a \emph{Hermitian cone}. The radical of a Hermitian cone,
that is  the set $\{Y \in \Sigma| \  Y H (X^{\sigma})^T=0 \ \forall
X \in \Sigma\}$, consists of one point, the \emph{vertex} of the
cone.

All non--degenerate Hermitian varieties are projectively equivalent;
a possible canonical equation is
\begin{equation}\label{eqp}
 X_1^{q+1}+\ldots+X_{n-1}^{q+1}+X_n^qX_{n+1}+X_nX_{n+1}^q=0,
\end{equation}
where the polynomial on the left side of \eqref{eqp} is a
\emph{Hermitian form}. All Hermitian cones of $\Sigma$ are also
projectively equivalent.

A non--degenerate Hermitian  variety $\cH(n,q^2)$
 of $\Sigma$ has several
remarkable properties, see \cite{Segre,H1}; here we just recall the
following.
\begin{enumerate}[(1)]
\item The number of points on $\cH(n,q^2)$ is
  \[\mu_n(q)=q^{2n-1}+q(q^{n-\epsilon}+\ldots+q^{2n-4})+
  q^{n+\epsilon-2}+\ldots+q^2+1,\]
  where $\epsilon=0$ or $1$, according as $n$ is
  even or odd.
\item
  A maximal subspace of $\Sigma$ included in $\cH(n,q^2)$
  has dimension
  \[\left\lfloor \frac{n-1}{2}\right\rfloor.\]
  These maximal subspaces are called \emph{generators} of
  $\cH(n,q^2)$.
\item Any line of $\Sigma$ meets $\cH(n,q^2)$ in $1$,  $q+1$  or $q^2+1$
  points. The lines meeting $\cH$ in one point are called
  \emph{tangent lines}.
\item The  polar hyperplane $\pi_P$ with respect to
  $h$ of a point $P$ on $\cH(n,q^2)$ is the locus of the lines through $P$
  either
  contained in $\cH(n,q^2)$ or  tangent to it at $P$.
  This hyperplane $\pi_P$ is also
  called the \emph{tangent hyperplane} at $P$ of $\cH(n,q^2)$. Furthermore,
  \[|\cH(n,q^2) \cap \pi_P |=
  1+q^2\mu_{n-2}(q).
  \]
\item Every hyperplane $\pi$ of $\Sigma$ which is not a tangent
  hyperplane of $\cH(n,q^2)$ meets $\cH(n,q^2)$ in a non--degenerate Hermitian
  variety $\cH(n-1,q^2)$ of $\pi$.
\end{enumerate}

In Section \ref{affdes} we shall make extensive use of
non--degenerate Hermitian varieties, together with Hermitian cones
of vertex the point $P_{\infty}(0,0,\ldots,1,0)$. Let
$AG(n,q^2)=\Sigma\setminus \Sigma_{\infty}$ be the affine space
embedded in $\Sigma$. We may  provide an affine representation for
the Hermitian cones with  vertex at $P_{\infty}$ as follows.

Let $\varepsilon$ be a primitive element of $\GF(q^2)$. Take a point
$(a_1,\ldots,a_{n-1},0)$ on  the affine hyperplane $\Pi: X_{n}=0$ of
$AG(n,q^2)$. We can always write $a_i=a_i^1+\varepsilon a_i^2$ for
any $i=1,\ldots,n-1$.
There is thus a bijective correspondence $\vartheta$ between the
points of $\Pi$ and those of $AG(2n-2,q)$,
\[ \vartheta(a_1,\ldots,a_{n-1},0)=
(a_1^1,a_1^2,\ldots,a_{n-1}^1,a_{n-1}^2).
\]
Pick now a hyperplane $\pi'$ in $AG(2n-2,q)$ and consider its
pre--image  $\pi=\vartheta^{-1}(\pi')$ in $\Pi$. The set of all the
lines $P_{\infty}X$ with $X\in\pi$ is a Hermitian cone of vertex
$P_{\infty}$. The set $\pi$ is a basis of this cone.

Let $T_0=\{t\in\GF(q^2): \tr(t)=0 \}$, where $\tr: x \in\GF(q^2)
\mapsto x^q+x \in\GF(q)$ is the trace function. Then, such an
Hermitian cone  $\cH_{\mathbf{\omega},v}$ is represented by
\begin{equation}\label{eqc}
\omega_1^qX_1-\omega_1X_1+\omega_2^qX_2^q-\omega_2X_2+\ldots+
\omega_{n-1}^qX_{n-1}^q-\omega_{n-1}X_{n-1}=v,
 \end{equation}
where $\omega_i \in\GF(q^2)$,  $v \in T_0$ and there exists at least one
$i\in\{1,\ldots,n-1\}$ such that $\omega_i\neq 0$.

\section{Construction}
\label{oarr}

In this section we provide a family of simple orthogonal arrays
$OA(q^{2n-1},q^{2n-2}, q,2)$, where  $n$ is  a positive integer and
$q$  is any prime power. Several constructions based on finite
fields of orthogonal arrays are known, see for instance \cite{Bo,
FuMi1,FuMi2}. The construction of
 \cite{Bo} is based upon linear
transformations over finite fields. Non--linear functions are used
in \cite{FuMi1,FuMi2}. In  \cite{FuMi2}, the authors dealt with a
subgroup of $PGL(4,q)$, in order to obtain suitable quadratic
functions in $4$ variables; then, the domain $\cW$ of these
functions was appropriately restricted, thus producing an orthogonal
array $OA(q^3,q^2,q,2)$. The construction used in the aforementioned
papers starts from
 $k$ distinct multivariate
functions $f_1,\ldots,f_k$, all
with a common domain $\cW\subseteq\GF(q)^{n+1}$,
which provide an array
\[
A(f_1,\ldots,f_k;\cW)= \left\{\begin{pmatrix}
    f_1(x)\\
    f_2(x)\\
    \vdots\\
    f_k(x)
\end{pmatrix} : x \in \cW \right\}, \]
with  an arbitrary order of columns.
\par
In general, it is possible to generate  functions $f_i$ starting
from homogeneous polynomials in $n+1$ variables and considering the
action of a subgroup of the projective group $PGL(n+1,q)$. Indeed,
any given homogeneous polynomial $f$ is associated to a variety
$V(f)$ in $\Sigma$ of equation
\[f(x_1,\ldots,x_{n+1})=0.\]
The image $V(f)^g$ of $V(f)$ under the action of an element $g\in
PGL(n+1,q)$ is a variety $V(f^g)$ of $\Sigma$, associated to the
polynomial $f^g$.

A necessary condition for $A(f_1,\ldots,f_k; \cW)$ to be an
orthogonal array, when all the $f_i$'s are homogeneous, is that
$|V(f_i)\cap V(f_k)\cap\cW|$ is independent of the choice of $i$,
$j$, whenever $i\neq j$.


Here, we consider homogeneous polynomials which are  Hermitian forms
of $\GF(q^2)[X_1,\ldots,X_n,X_{n+1}]$. Denote by $G$ the subgroup of
$PGL(n+1,q^2)$ consisting of all collineations represented by
\[\alpha (X'_1,\ldots,X'_{n+1})= (X_1,\ldots, X_{n+1})M\]
where $\alpha \in GF(q^2)\setminus\{0\}$, and
\begin{equation} \label{collin}
M=\begin{pmatrix}
    1 &0& \ldots & 0 &j_1& 0\\
    0 & 1&\ldots & 0 & j_2& 0\\
    \vdots&&&&&\vdots\\
    0 & 0&&1 & j_{n-1}&0 \\
    0& 0&\ldots &0 & 1 & 0\\
    i_1 &i_2& \ldots&i_{n-1} & i_n&1 \\
  \end{pmatrix}^{-1},
\end{equation}
with $i_s, \ j_{m}\in\GF(q^2)$. The group $G$ has order
$q^{2(2n-1)}$. It stabilises the hyperplane $\Sigma_\infty$, fixes
the point $P_{\infty}(0,\ldots,0,1,0)$ and acts transitively on
$AG(n,q^2)$.

Let $\cH$ be the non--degenerate Hermitian  variety associated to
the Hermitian form \[F=
X_1^{q+1}+\ldots+X_{n-1}^{q+1}+X_n^qX_{n+1}+X_nX_{n+1}^q.\] The
hyperplane $\Sigma_\infty$ is the tangent hyperplane at $P_\infty$
of $\cH$. The Hermitian form associated to the variety $\cH^g$, as
$g$ varies in $G$, is
\begin{equation}
\label{fgeq}
\begin{split}
F^g=
X_1^{q+1}&+\ldots+X_{n-1}^{q+1}+X_{n}^qX_{n+1}+X_{n}X_{n+1}^q+
X_{n+1}^{q+1}(i_1^{q+1}+\ldots+i_{n-1}^{q+1}+i_n^q+i_n)
\\
&
+\tr\left(X_{n+1}^q(X_1(i_1^q+j_1)+\ldots+X_{n-1}(i_{n-1}^q+j_{n-1}))\right)
\end{split}
\end{equation}
The subgroup $\Psi$ of $G$ preserving $\cH$ consists of all
collineations whose matrices satisfy the condition
\[\left\{\begin{array}{l}
     j_1=-i_1^q \\
     \vdots\\
     j_{n-1}=-i_{n-1}^q\\
     i_1^{q+1}+\ldots+i_{n-1}^{q+1}+i_n^q+i_n=0
   \end{array}\right.. \]
Thus, $\Psi$ contains  $q^{(2n-1)}$ collineations and acts on the
affine points of $\cH$ as a sharply transitive permutation group.
Let $C=\{a_1=0,\ldots,a_q\}$ be a system of representatives for the
cosets  of $T_0$, viewed as an additive subgroup of $\GF(q^2)$.
Furthermore, let
 $\cR$ denote the subset of $G$
 whose
collineations are induced by
\begin{equation}\label{rep}
 M'= \begin{pmatrix}
    1 &0& \ldots & 0 &0& 0\\
    0 & 1&\ldots & 0 & 0& 0\\
    \vdots&&&&&\vdots\\
    0 & 0&&1 &0&0 \\
    0& 0&\ldots &0 & 1 & 0\\
    i_1& i_2& \ldots&i_{n-1} & i_n&1 \\
  \end{pmatrix}^{-1},
\end{equation}
where $i_1,\ldots,i_{n-1} \in GF(q^2)$, and for each tuple
$(i_1,\ldots,i_{n-1})$,  the element $i_n$ is the unique solution in
$C$ of equation
\begin{equation}\label{ara}
i_1^{q+1}+\ldots+i_{n-1}^{q+1}+i_n^q+i_n=0.
\end{equation}
The set  $\cR$ has cardinality $q^{2n-2}$ and can be used to
construct a set of Hermitian form $\{F^g| g\in \cR\}$ whose related
 varieties are pairwise distinct.
%
%
%
%
\begin{theorem}
  \label{teo:oa1}
  For any given prime power $q$, the matrix $\cA=A(F^g,g \in \cR; \cW)$,
  where
  \[\cW=\{(x_1,\ldots,x_{n+1})\in\GF(q^2)^{n+1}:
  x_{n+1}=1\}, \]
  is
  an $OA(q^{2n},q^{2n-2},q,2)$ of index $\mu=q^{2n-2}$.
\end{theorem}
\begin{proof}
It is sufficient to
  show that the number of solutions in $\cW$ to the system
  \begin{equation}
    \left\{\begin{array}{l}
        \label{orto1}
        F(X_1,X_2,\ldots,X_n,X_{n+1})=\alpha \\
        F^g(X_1,X_2,\ldots,X_n,X_{n+1})=\beta
      \end{array}\right.
  \end{equation}
is $q^{2n-2}$ for any $\alpha$, $\beta \in\GF(q)$, $g\in
  \cR\setminus\{id\}$.
  By definition of  $\cW$, this
system  is equivalent to
  \begin{equation}
    \left\{\begin{array}{l}
        \label{orto2}
        X_1^{q+1}+\ldots+X_{n-1}^{q+1}+X_n^q+X_n=\alpha  \\
        X_1^{q+1}+\ldots+X_{n-1}^{q+1}+X_{n}^q+X_{n}+
        \tr\left(X_1i_1^q+\ldots+X_{n-1}i_{n-1}^q\right)=\beta
      \end{array}\right.
\end{equation}
Subtracting the first equation from the second we get
\begin{equation}\label{trc:2}
 \tr(X_1i_1^q+\ldots+X_{n-1}i_{n-1}^q)=\gamma,
\end{equation}
where $\gamma=\beta-\alpha$. Since  $g$ in not the identity then,
$(i_1^q, \ldots,i_{n-1}^q) \neq (0,\ldots,0)$, and hence Equation
\eqref{trc:2} is equivalent to the union of $q$ linear equations
over $\GF(q^2)$ in $X_1, \ldots, X_{n-1}$. Thus, there are
$q^{2n-3}$ tuples $(X_1, \ldots, X_{n-1})$ satisfying \eqref{trc:2}.
For each such a tuple, \eqref{orto2} has $q$ solutions in $X_n$ that
provide a coset of $T_0$ in $\GF(q^2)$. Therefore, the system
\eqref{orto1} has $q^{2n-2}$ solutions in $\cW$ and the result
follows.
%
\end{proof}
The array $\cA$ of Theorem \ref{teo:oa1} is not simple since
\begin{equation}\label{col} F^{g}(x_1,\ldots,x_n,1)=
F^{g}(x_1,\ldots, x_n+r,1)
\end{equation}
 for any $g \in \cR$, and
$r\in T_0$.

We now investigate how to extract a subarray $\cA_0$ of $\cA$ which
is simple.
 We shall need a preliminary lemma.
 \begin{lemma}
   \label{lem:tr}
   Let $x\in\GF(q^2)$ and suppose
   $\tr(\alpha x)=0$ for any $\alpha\in\GF(q^2)$.
   Then, $x=0$.
\end{lemma}
\begin{proof}
Consider $\GF(q^2)$ as a $2$--dimensional vector space over
$\GF(q)$. By \cite[Theorem 2.24]{LR}, for any  linear mapping
$\Xi:\GF(q^2)\rightarrow\GF(q)$, there exists exactly one
$\alpha\in\GF(q^2)$ such that $\Xi(x)=\tr(\alpha x)$. In particular,
if $\tr(\alpha x)=0$ for any $\alpha\in\GF(q^2)$, then $x$, is in
the kernel of all linear mappings $\Xi$. It follows that $x=0$.
%
%
%
\end{proof}
\begin{theorem}
\label{teo:fh0}
 For any prime power $q$, the matrix
 $\cA_0=A(F^g,g\in \cR,\cW_0)$, where
 \[ \cW_0=\{(x_1,\ldots,x_{n+1})\in \cW: x_n\in C\} \]
 is a simple $OA(q^{2n-1},q^{2n-2},q,2)$
 of index $\mu=q^{2n-3}$.
\end{theorem}
\begin{proof}
  We first show that $\cA_0$ does not contain any
  repeated column. Let $\cA$ be the array introduced
  in Theorem \ref{teo:oa1}, and
  index its columns by the
  corresponding elements in $\cW$.
  Observe that the
  column
  $(x_1,\ldots,x_n,1)$
  is the same as
  $(y_1,\ldots,y_n,1)$
  in $A$
  if, and only if,
  \[ F^g(x_1,\ldots,x_n,1)=F^g(y_1,\ldots,y_n,1), \]
  for any $g\in \cR$.
  We thus obtain a system of $q^{2n-2}$
  equations in the $2n$ indeterminates $x_1,\ldots
  x_n,y_1,\ldots,y_n$. Each equation is of the form
  \begin{equation}
    \label{eq:scol}
    \tr(x_n-y_n)=
    \sum_{t=1}^{n-1}\left(y_t^{q+1}-x_t^{q+1}+\tr(a_t(y_t-x_t))\right),
  \end{equation}
  where the elements $a_t=i_t^q$ vary in $\GF(q^2)$ in
  all possible ways.
  The left hand side of the equations in \eqref{eq:scol}
  does not depend on the elements $a_t$; in particular,
  for $a_1=a_2=\ldots=a_t=0$ we have,
  \[ \tr(x_n-y_n)=\sum_{t=1}^{n-1}(y_t^{q+1}-x_t^{q+1}); \]
  hence,
  \[
 \sum_{t=1}^{n-1}(y_t^{q+1}-x_t^{q+1})=
 \sum_{t=1}^{n-1}\left(y_t^{q+1}-x_t^{q+1}+\tr(a_t(y_t-x_t))\right)
   \]
  Thus,
  $\sum_{t=1}^{n-1}\tr(a_t(y_t-x_t))=0$.
  By the arbitrariness of the coefficients $a_t\in\GF(q^2)$,
  we obtain that for any $t=1,\ldots n-1$, and any
  $\alpha\in\GF(q^2)$,
  \[ \tr(\alpha(y_t-x_t))=0. \]
  Lemma \ref{lem:tr} now yields
  $x_t=y_t$ for any $t=1,\ldots, n-1$ and
  we also get from
  \eqref{eq:scol}
  \[ \tr(x_n-y_n)=0. \]
  Thus, $x_n$ and $y_n$ are in the same coset of
  $T_0$. It follows that two columns of $\cA$
  are the same if and only if the difference of their
  indexes in $\cW$ is
  a vector of the form $(0,0,0,\ldots,0,r,0)$
  with $r\in T_0$. By construction, there are
  no two distinct vectors in $\cW_0$ whose difference
  is of the required form; thus, $\cA_0$ does
  not contain repeated columns.

  The preceding argument shows that the
  columns of $\cA$ are partitioned
  into $q^{2n-1}$ classes, each consisting of $q$ repeated columns.
  Since $\cA_0$ is obtained from $\cA$ by deletion
  of $q-1$ columns in each class, it follows that $\cA_0$ is an
  $OA(q^{2n-1},q^{2n-2},q,2)$ of index $q^{2n-3}$.
  \end{proof}

\section{A non--classical model of $AG(2n-1,q)$}
 \label{affdes}

We keep the notation introduced in the previous sections. We are
going to  construct an affine
$2-\left(q^{2n-1},q^{2n-2},q^{(2n-3)}+\ldots+q+1\right)$ design
$\cS$ that, as we will see, is related to the array $A_0$ defined in
Theorem \ref{teo:fh0}. Our construction is a generalisation of
\cite{A}.


Let again consider the subgroup $G$ of $PGL(n+1,q^2)$ whose
collineations are induced by matrices  \eqref{collin}.
The group $G$
acts on the set
of all Hermitian cones  of the form \eqref{eqc} as
a permutation group.
In this action,
${G}$ has $q^{(2n-3)}+\ldots+1$ orbits, each of size $q$.
In particular the $q^{(2n-3)}+\ldots+1$ Hermitian cones
$\cH_{\mathbf{\omega},0}$
 of affine equation
\begin{equation} \label{eqc2}
\omega_1^qX_1-\omega_1X_1+\omega_2^qX_2^q-\omega_2X_2+\ldots+
\omega_{n-1}^qX_{n-1}^q-\omega_{n-1}X_{n-1}=0,
 \end{equation}
with $(\omega_1,\ldots,\omega_{n-1}) \in
GF(q^2)^{n-1}\setminus\{(0,\ldots,0)\}$, constitute a system of
representatives for these orbits.

The stabiliser in $G$ of the origin $O(0,\ldots,0,1)$  fixes the
line $OP_{\infty}$ point--wise, while is transitive on the points of
each other line passing through $P_{\infty}$. Furthermore, the
centre of $G$ comprises all collineations induced by
\begin{equation}\label{centre} \begin{pmatrix}
    1 &0& \ldots & 0 &0& 0\\
    0 & 1&\ldots & 0 & 0& 0\\
    \vdots&&&&&\vdots\\
    0 & 0&&1 &0&0 \\
    0& 0&\ldots &0 & 1 & 0\\
    0&0& \ldots&0 & i_n&1 \\
  \end{pmatrix}^{-1}, \end{equation}
with $i_n \in GF(q^2)$. The subset of \eqref{centre} with $i_n\in
T_0$ induces a normal subgroup $N$ of $G$ acting semiregularly on
the affine points of $AG(n,q^2)$ and preserving each line parallel
to the $X_n$-axis. Furthermore, $N$ is contained in $\Psi$ and also
preserves every affine Hermitian cone $\cH_{\mathbf{\omega},v}$.

We may now define an incidence structure $\cS=(\cP,\cB,I)$ as
follows. The set $\cP$ consists of all the point--orbits of
$AG(n,q^2)$ under the action of $N$. Write $N(x_1,\ldots,x_n)$ for
the  orbit  of the point $(x_1,\ldots,x_n)$ in $AG(n,q^2)$under the
action of $N$.

The elements of $\cB$ are the images of the Hermitian variety $\cH$
of affine equation
\begin{equation}\label{affeq1}X_1^{q+1}+\ldots+X_{n-1}^{q+1}+X_n^q+X_n=0,
\end{equation} together with the
images of the Hermitian cones \eqref{eqc2} under the action of  $G$.
If a block $B\in\cB$ arises from \eqref{affeq1}, then it will be
called \emph{Hermitian--type}, whereas if $B$ arises from
\eqref{eqc2}, it will be \emph{cone--type}. Incidence is given by
inclusion.

\begin{theorem}\label{mainthm}
The aforementioned incidence structure $\cS$ is an affine
\[ 2-(q^{(2n-1)},q^{2(n-1)},q^{(2n-3)}+\ldots+q+1) \]
design, isomorphic, for $q>2$,
 to the point--hyperplane design  of the affine space $AG(2n-1,q)$.
\end{theorem}
\begin{proof}
By construction, $\cS$ has $q^{2n-1}$ points and
$q^{(2n-1)}+q^{2(n-1)}\ldots+q$  blocks, each block consisting of
$q^{2(n-1)}$ points.

We first prove that  the number of blocks through any two given
points is $q^{(2n-3)}+\ldots+q+1$. Since $\cS$ has a
point--transitive automorphism group, we may assume, without loss of
generality, one of these points to be $O=N(0,\ldots,0)$. Let
$A=N(x_1,x_2,\ldots,x_{n})$ be the other point.
We distinguish two cases,
according as the points lie on the same line
through $P_{\infty}$ or not.

We begin by considering the case
 $(0,0,\ldots,0)\neq(x_1,x_2,\ldots,x_{n-1})$.
The line $\ell$
represented by $X_1=x_1, \ldots, X_{n-1}=x_{n-1}$, is a  secant to
the Hermitian variety $\cH$. Since the stabiliser of the origin is
transitive on the points of $\ell$, we may assume that
$A\subseteq\cH$; in particular,
$(x_1,x_2,\ldots,x_{n})\in\cH$ and
\begin{equation}\label{eqherm}
x_1^{q+1}+\ldots+x_{n-1}^{q+1}+x_{n}^q+x_{n}=0.
\end{equation}
Observe that this condition is satisfied by every possible
representative of $A$. Another Hermitian type block arising from the
variety $\cH^g $ associated to the form \eqref{fgeq}, contains the
points $O$ and $A$ if and only if
\begin{equation}\label{eq1her1}
i_1^{q+1}+\ldots+i_{n-1}^{q+1}+i_n^q+i_n=0
\end{equation}
and
\begin{equation}\label{eq2her1}
\begin{array}{c}x_1^{q+1}+\ldots+x_{n-1}^{q+1}+x_{n}^q+x_{n}+
  x_1^q(i_1+j_1^q)+\ldots+x_{n-1}^q(i_{n-1}+j_{n-1}^q)+\\
+x_1(i_1^q+j_1)+\ldots+x_{n-1}(i_{n-1}^q+j_{n-1})+i_1^{q+1}+
\ldots+i_{n-1}^{q+1}+i_n^q+i_n=0.
\end{array}
\end{equation}
 Given \eqref{eqherm} and \eqref{eq1her1}, Equation \eqref{eq2her1}
 becomes
 \begin{equation}\label{trc}
 \tr(x_1(i_1^q+j_1)+\ldots+x_{n-1}(i_{n-1}^q+j_{n-1}))=0
\end{equation}
Condition \eqref{eq1her1}
shows that there are $q^{2n-1}$ possible choices for
the tuples $\mathbf{i}=(i_1, \ldots, i_n)$; for any such a tuple,
because of \eqref{trc},
we get $q^{2n-3}$ values for $\mathbf{j}=(j_1,\ldots,j_{n-1})$.
Therefore, the total number of
Hermitian--type blocks through the points $O$ and $A$ is
exactly
\[\frac{q^{4(n-1)}}{q^{2n-1}}=q^{2n-3}.\]
On the other hand, cone--type blocks containing $O$ and $A$ are just
cones with basis a  hyperplane of $AG(2n-2,q)$, through the line
joining   the affine points $(0,\ldots,0)$ and
$\theta(x_1,\ldots,x_{n-1},0)$; hence, there are precisely
$q^{2n-4}+\ldots+q+1$ of them.

We now deal with the case $(x_1,x_2,\ldots,x_{n-1})= (0,0,\ldots,
0)$. A Hermitian--type block through  $(0,\ldots,0)$ meets the
$X_n$-axis at  points of the form $(0,\ldots,0,r)$ with $r\in T_0$.
Since $x_{n} \notin T_0$,  no Hermitian--type block may contain both
$O$ and $A$. On the other hand, there are $q^{2n-3}+\ldots+q+1$
cone-type blocks through the two given points that is, all cones
with basis a  hyperplane in $AG(2n-2,q)$ containing the origin of
the reference system in $AG(2n-2,q)$. It follows that $\cS$ is a
$2-(q^{(2n-1)}, q^{2(n-1)}, q^{(2n-3)}+\ldots+q+1)$ design.

Now we recall that two blocks of a design may be defined  parallel
if they are either coincident or disjoint. In order to show that
$\cS$ is indeed an affine design we need to check the following two
properties, see \cite[Section 2.2, page 72]{DE}:
\begin{enumerate}[(a)]
\item\label{prop:a} any two distinct blocks either are disjoint or have
$q^{2n-3}$ points in common;
\item\label{prop:b} given a point $N(x_1,\ldots,x_n) \in \cP$ and a block
$B\in\cB$ such that $N(x_1,\ldots,x_n)\notin B$, there exists a
unique block $B'\in \cB$ satisfying both $N(x_1,\ldots,x_n)\in B'$
and $B \cap B'=\emptyset$.
\end{enumerate}
We start by showing that  (\ref{prop:a})  holds for any two distinct
Hermitian--type blocks. As before, we may suppose one of them to be
$\cH$ and denote by $\cH^g$ the other one, associated to the form
\eqref{fgeq}. We need to solve the system  of equations given by
\eqref{eqherm} and \eqref{eq2her1}. Subtracting \eqref{eqherm} from
\eqref{eq2her1},
\begin{equation}
  \label{trc:1}
  \tr(x_1(i_1^q+j_1)+\ldots+x_{n-1}(i_{n-1}^q+j_{n-1}))=\gamma,
\end{equation}
where $\gamma=-(i_1^{q+1}+\ldots+i_{n-1}^{q+1}+i_n^q+i_n)$.

Suppose that  $(i_1^q+j_1, \ldots,i_{n-1}^q+j_{n-1}) \neq (0,\ldots,0)$.
Arguing as in the proof of
Theorem \ref{teo:oa1}, we see that there are $q^{2n-3}$ tuples
$(x_1, \ldots, x_{n-1})$ satisfying \eqref{trc:1} and,  for each
such a tuple, \eqref{eqherm} has $q$ solutions in $x_1$. Thus, the
system given by \eqref{eqherm} and \eqref{eq2her1} has $q^{2n-2}$
solutions; taking into account the definition of a point of $\cS$,
it follows that the number of the common points of the two blocks
under consideration is indeed $q^{2n-3}$.

In the case $(i_1^q+j_1, \ldots,i_{n-1}^q+j_{n-1})= (0,\ldots,0)$,
either $\gamma \neq 0$ and the two blocks are disjoint, or
$\gamma=0$ and the two blocks are the same.

We now move to consider the case wherein both blocks are cone--type.
The  bases of these blocks are either disjoint or
share $q^{2(n-2)}$ affine
points; in the former case, the blocks are disjoint;
in the latter,
they have $q^{2(n-2)}$
lines in common. Since each line of $AG(n,q^2)$
consists of $q$
points of $\cS$, the intersection of the two blocks has size
$q^{2n-3}$.

We finally study the intersection of two blocks of different type.
We may assume again the Hermitian--type block to be $\cH$. Let then
$\cC$ be  cone--type.  Each generator of  $\cC$ meets the Hermitian
variety $\cH$ in $q$ points which form an orbit of $N$. Therefore,
the number of common points between the two blocks is, as before,
$q^{2n-3}$; this completes the proof of (\ref{prop:a}).

We are going to show that property (\ref{prop:b}) is also satisfied.
By construction, any cone--type block meets every Hermitian--type
block. Assume first $B$ to be the Hermitian variety $\cH$ and
$P=N(x_1,x_2,\ldots,x_n)\not\subseteq\cH$. Since we are looking for
a block $B'$  through $P$, disjoint from $\cH$, also $B'$ must be
Hermitian--type. Let $\beta$ be the collineation induced by
\[
\begin{pmatrix}
    1 &0& \ldots & 0 &0& 0\\
0 & 1&\ldots & 0 & 0& 0\\
\vdots&&&&&\vdots\\
    0 & 0&&1 & 0&0 \\
     0& 0&\ldots &0 & 1 & 0\\
    0 &0& \ldots& 0& i_n&1 \\
      \end{pmatrix}^{-1}, \]
with $i_n^q+i_n+x_1^{q+1}+ \ldots +x_{n-1}^{q+1}+x_n^q+x_n=0$. Then,
the image $B'$ of $\cH$ under $\beta$ is disjoint from $\cH$ and
contains the set $P$. To prove the uniqueness of the block
satisfying condition (\ref{prop:b}), assume that there is another block
$\widetilde{B}$, which is the image of $\cH$ under the collineation
$\omega$ induced by
\[\begin{pmatrix}
    1 &0& \ldots & 0 &b_1& 0\\
0 & 1&\ldots & 0 & b_2& 0\\
\vdots&&&&&\vdots\\
    0 & 0&&1 & b_{n-1}&0 \\
     0& 0&\ldots &0 & 1 & 0\\
    a_1 &a_2& \ldots&a_{n-1} & a_n&1 \\
      \end{pmatrix}^{-1}, \]
and
such that $\widetilde{B} \cap \cH= \emptyset$ and $P\subseteq\widetilde{B}$.
 As $\widetilde{B}$ and $\cH$ are disjoint,
 the  system given by \eqref{eqherm} and
 \begin{equation}\label{eq3her1}
   \begin{array}{c}x_1^{q+1}+\ldots+x_{n-1}^{q+1}+x_{n}^q+x_{n}+x_1^q(a_1+b_1^q)+\ldots+x_{n-1}^q(a_{n-1}+b_{n-1}^q)+\\
     +x_1(a_1^q+b_1)+\ldots+x_{n-1}(a_{n-1}^q+b_{n-1})
     +a_1^{q+1}+\ldots+a_{n-1}^{q+1}+a_n^q+a_n=0.
\end{array}
\end{equation}
must have no solution. Arguing as in the proof of (\ref{prop:a}),
we see that this
implies that $(a_1^q+b_1, \ldots,a_{n-1}^q+b_{n-1})= (0,\ldots,0)$.
On the other hand, $P\in \widetilde{B}\cap B'$ yields
$i_n^q+i_n+a_1^{q+1}+\ldots+a_{n-1}^{q+1}+a_n^q+a_n=0$, that is $
\omega^{-1}\beta$ is in the stabiliser $\Psi$ of $\cH$ in $G$;
hence, $B'=\widetilde{B}$.

Now, assume $B$ to be a cone-type block. Denote by  $\pi$  its basis
and let $P'=(x_1^1,x_1^2,\ldots,x_{n-1}^1,x_{n-1}^2)$ be the image
$\vartheta(x_1,\ldots,x_{n-1},0)$ on the affine space $AG(2n-2,q)$
identified, via $\vartheta$, with the affine hyperplane $X_n=0$. In
$AG(2n-2,q)$ there is a unique hyperplane $\pi'$ passing trough the
point $P'$ and disjoint from $\pi$. This hyperplane  $\pi'$ uniquely
determines the block $B'$ with property (\ref{prop:b}).

In order to conclude the proof of the current theorem we shall
require a deep characterisation of the high--dimensional affine
space, namely that an affine design $\cS$ such that $q>2$,  is an
affine space if and only if every line consists of exactly $q$
points, see \cite[Theorem 12, p. 74]{DE}.

Recall that the line of a design $\cD$ through two given points
$L,M$ is defined as the set of all points of $\cD$  incident to
every block containing both $L$ and $M$.  Thus, choose two distinct
points in $\cS$. As before, we may assume that one of them is
$O=N(0, \ldots, 0)$ and let $A=N(x_1, \ldots, x_n)$ be the other
one.

Suppose first that $A$  lies on the $X_n$-axis. In this case, as we
have seen before, there are exactly $q^{2n-3}+\ldots+q+1$ blocks
incident to both $O$ and $A$, each of them  cone--type. Their
intersection consists of $q$ points of $\cS$ on the $X_n$-axis.

We now examine the case where $A$ is not on the $X_n$-axis. As
before, we may assume that $A\subseteq\cH$, hence \eqref{eqherm}
holds.
 Exactly
 $q^{2n-3}+\ldots+q+1$ blocks  are incident to both $O$ and $A$:
  $q^{2n-2}$ are  Hermitian--type, the remaining
 $q^{2n-4}+\ldots+q+1$ being cone--type.
 Hermitian--type blocks
 passing through $O$ and $A$ are represented by
\begin{equation}\label{ints}
\begin{array}{c}X_1^{q+1}+\ldots+X_{n-1}^{q+1}+X_{n}^q+X_{n}+
  X_1^q(i_1+j_1^q)+\ldots+\\X_{n-1}^q(i_{n-1}+j_{n-1}^q)
+X_1(i_1^q+j_1)+\ldots+X_{n-1}(i_{n-1}^q+j_{n-1})=0,
\end{array}
\end{equation}
with \eqref{trc} satisfied. Set $x_s=x_s^1+\varepsilon x_s^2$ for
any $s=1,\ldots,n$, with $x_s^1,x_s^2\in\GF(q)$.
 The cone--type
blocks  incident to both $O$ and $A$  are  exactly those with basis
a hyperplane of $AG(2n-2,q)$ containing the line through the points
$(0,\ldots,0)$ and $(x_1^1,x_1^2,\ldots,x_{n-1}^1,x_{n-1}^2)$.
Hence, these blocks share $q$ generators, say $r_t$, with affine
equations of the form
\[
r_t\left\{\begin{array}{l}
     X_1=tx_1 \\
     \vdots\\
     X_{n-1}=tx_{n-1}
   \end{array}\right.
 \]
as $t$ ranges over $\GF(q)$.  Each generator $r_t$ meets the
intersection of the Hermitian--type blocks through $O$ and $A$ at
those  points $(tx_1,tx_2,\ldots,tx_{n-1}, \overline{x}_n) $
satisfying each of the \eqref{ints}, that is
\begin{equation}\label{eq4}
  \begin{array}{c}
    t^2x_1^{q+1}+\ldots+t^2x_{n-1}^{q+1}+\overline{x}_n^q+\overline{x}_n+tx_1^q(i_1+j_1^q)+
    \ldots+tx_{n-1}^q(i_{n-1}+j_{n-1}^q)+\\
    +tx_1(i_1^q+j_1)+\ldots+tx_{n-1}(i_{n-1}^q+j_{n-1})=0.
\end{array}\end{equation}
Given \eqref{eqherm}, \eqref{trc}, Equations  \eqref{eq4} become
\begin{equation}\label{eq5}
\overline{x}_n^q+\overline{x}_n-t^2(x_n^q+x_n)=0.
\end{equation}
Since $t^2(x_n^q+x_n)\in\GF(q)$, \eqref{eq5} has $q$ solutions, all
of the form $\{\overline{x}_{n}+r | r \in T_0 \}$. The point-set
$\{(tx_1,tx_2,\ldots,tx_{n-1}, \overline{x}_n+r)| r\in T_0\} $
coincides with the point $N(tx_1,tx_2, \ldots,tx_{n-1},
\overline{x}_{n}) \in \cP$ and  as $t$ varies in $GF(q)$, we get
that the intersection of all blocks containing $O$ and $A$ consists,
also in this case, of $q$ points of $\cS$.
\end{proof}
\begin{remark} The array  $A_0$ defined in Theorem \ref{teo:fh0} is closely related to
the affine design $\cS=(\cP,\cB,I)$. Precisely, $\cW_0$ is a set of
representatives for $\cP$. The rows of $\cA_0$ are generated by the
forms  $F^g$ for $g$ varying  in $\cR$, whose associated Hermitian
varieties provide a set of representatives for the $q^{2n-2}$
parallel classes of Hermitian--type blocks in $\cB$.
\end{remark}

\bigskip
\begin{minipage}{6.8cm}
\begin{obeylines}
{\sc Angela Aguglia}
Dipartimento di Matematica
Politecnico di Bari
Via G. Amendola 126/B
70126 Bari
Italy
{\tt a.aguglia@poliba.it}
\end{obeylines}
\end{minipage}
\qquad
\begin{minipage}{6.8cm}
\begin{obeylines}
{\sc Luca Giuzzi}
Dipartimento di Matematica
Politecnico di Bari
Via G. Amendola 126/B
70126 Bari
Italy
{\tt l.giuzzi@poliba.it}
\end{obeylines}
\end{minipage}

\end{document}